\documentclass[11pt]{amsart} 
\usepackage{amsfonts} 
\usepackage{graphicx}
\usepackage{color}
\usepackage{pinlabel}
\usepackage[colorlinks=true, linkcolor=blue, citecolor=blue]{hyperref}

\author{Mokhtar Aouina}
\address{Western Illinois University\\
Macomb, Illinois, 61438}
\email{m-aouina@wiu.edu}
\author{Hamed Karami}
\address{Georgia State University\\
Atlanta, Georgia, 30303}
\email{hkarami1@student.gsu.edu}
\author{Douglas J. LaFountain}
\address{Western Illinois University\\
Macomb, Illinois, 61438}
\email{d-lafountain@wiu.edu}

\newtheorem{theorem}{Theorem}[section]
\newtheorem{proposition}[theorem]{Proposition}
\newtheorem{lemma}[theorem]{Lemma}

\theoremstyle{definition}

\newtheorem{remark}[theorem]{Remark}

\renewenvironment{proof}{{\noindent\bf Proof.}}{\hfill $\Box$\par\vskip3mm}

\begin{document}

\title{Distance-balancing of cube-connected cycles graphs}
\maketitle

\begin{abstract}
We show that cube-connected cycles graphs $CCC_n$ are distance-balanced, and nicely distance-balanced if and only if $n$ is even.
\end{abstract}

\section{Introduction}
The concept of distance-balanced graphs was implicitly introduced by Handa, where he conjectured that a bipartite distance-balanced graph which is not a cycle must be 3-connected \cite{1}. However, this term itself is due to the work of Jerebic, Klav\v{z}ar and Rall, who considered distance-balanced graphs in terms of different graph products \cite{2}. Since then, much work has been done to find interesting properties of these graphs \cite{5,7,3}, including to see if a given family of graphs is distance-balanced or not.  This line of inquiry has previously been fruitful for generalized Petersen graphs $GP(n,k)$ and Hamming graphs $H(n,q)$ \cite{8,4,9}.  For example, Jerebic, Klav\v{z}ar and Rall originally conjectured that for any $k \geq 2$ there exists an $N$ for which $GP(n,k)$ is not a distance-balanced graph for all $n \geq N$ \cite{2}; this conjecture was proved by Yang, Hou, Li and Zhong \cite{16}. In addition, Frelih generalized the concept of distance-balanced graphs to  $\ell$-distance-balanced graphs where $\ell$ is a positive integer and a lower bound for the diameter of the graph \cite{19}.  Using this definition, Ma, Wang and Klav\v{z}ar partially solved a conjecture of Miklavi\v{c} and Sparl \cite{18} by proving that for large $n$ compared to $k$, the generalized Petersen graph $GP(n, k)$ is $\textrm{diam}(GP(n, k))$-distance-balanced \cite{17}.

Our paper continues in this direction by studying a given family of graphs to see if it is distance-balanced or not, as well as an additional property, first introduced by Kutnar and Miklavi\v{c}, of being nicely distance-balanced \cite{3}.  Specifically, we consider the distance-balanced and nicely distanced-balanced properties of cube-connected cycles graphs, denoted $CCC_n$.  Cube-connected cycles are a family of graphs introduced by Preparata and Vuillemin for use as a network topology in parallel computing \cite{11}. These graphs are 3-regular, connected and Cayley \cite{12}, and properties such as their diameter and the existence of cycles of all lengths have been studied \cite{14, 13, 15}.  In this note, we show that all cube-connected cycles graphs are distance-balanced, but nicely distanced-balanced if and only if they are derived from a hypercube in even dimensions.  Since distance-regular graphs are nicely distance-balanced \cite{6,3}, this also shows that cube-connected cycles graphs are not distance-regular when formed from hypercubes in odd dimensions.

The structure of the paper is as follows.  In Section \ref{sec:prelim} we review the definitions of our objects of interest, and also establish properties of the automorphism group of cube-connected cycles that allow us to conclude in Proposition \ref{prop:db} that $CCC_n$ is distance-balanced for all $n$.  In Section \ref{sec:one} we then introduce a new way to represent paths in $CCC_n$, which is then applied in Section \ref{sec:two} to prove our main Theorem \ref{thm:ndb} that $CCC_n$ is nicely distance-balanced if and only if $n$ is even.

\section{Distance-balancing in $CCC_n$}
\label{sec:prelim}

We begin with the definitions of the objects and properties of interest.  Let $G$ be a finite graph with vertices $V$ and edges $E$.  Given an edge $uv \in E$, let $W_{u,v} = \{ x \in V \ : \ d(x,u) < d(x,v)\}$, and similarly  $W_{v,u} = \{ x \in V \ : \ d(x,v) < d(x,u)\}$.  Also denote $W^v_u = \{ x \in V \ : \ d(x,u) = d(x,v)\}$.  Then $G$ is said to be {\em distance-balanced} if for any edge $uv \in E$,  $|W_{u,v}| = |W_{v,u}|$.  $G$ is said to be {\em nicely distance-balanced} if there exists a constant $k \in \mathbb{N}$ such that for any edge $uv \in E$,  $|W_{u,v}| = |W_{v,u}|=k$.  

The {\em hypercube graph} $Q_n$ is the graph whose vertices and edges correspond combinatorially to those of the unit hypercube in $\mathbb{R}^n$.  Notationally, the vertices of $Q_n$ are binary strings of length $n$, and edges connect two vertices whose strings differ in just one digit.  The labeling of vertices of $Q_n$ is not unique; in particular, we have the following initial lemma which is common knowledge, but which nevertheless will be useful in what follows.

\begin{lemma}
\label{lem:labelcube}
A choice of vertex $u$ in $Q_n$ as $00\cdots0$, along with a choice of labeling for the $n$ vertices adjacent to $u$, determines the labeling of all vertices in $Q_n$.
\end{lemma}

\begin{proof}
The automorphism group Aut($Q_n$) has order $2^nn!$ (see for example \cite{21}), and each automorphism corresponds to a distinct labeling of the vertices of $Q_n$.   Specifically, there are $2^n$ vertices in $Q_n$, and by Lemma 3.1.1 in \cite{20} there is a subgroup $H$ of Aut$(Q_n)$, where $|H|=2^n$ and each element of $H$ is translation by a fixed binary $n$-string; this establishes vertex-transitivity of $Q_n$, and allows us to label any vertex as $00\cdots0$.  Moreover, the discussion following Lemma 3.1.1 in \cite{20} identifies another subgroup $K$ of Aut$(Q_n)$ which consists of permutations of digits in all $n$-strings, so that $|K|=n!$; this corresponds to the $n!$ permutations of the vertices $100\cdots00, 010\cdots00, \dots, 000\cdots01$ adjacent to $00\cdots 0$.   Since $H \cap K$ is just the identity element, it follows that Aut$(Q_n) = HK$, so that the choice of a vertex as $00 \cdots 0$ along with a choice of a labeling of its $n$ neighboring vertices determines the labeling of the entire graph.
\end{proof}

The {\em cube-connected cycles graph} $CCC_n$ is the graph where each vertex of the hypercube graph $Q_n$ is replaced by an $n$-cycle.  We can label vertices in $CCC_n$ by adding an extra digit to the binary strings of $Q_n$ in order to keep track of the $n$ vertices within each cycle.  Specifically, in $CCC_n$ vertices are $(n+1)$-strings of the form $x_1x_2 \dots x_nk$, where each $x_i \in \{0,1\}$ , $k \in \{1,\dots,n\}$, and two vertices are equal if and only if each digit in their strings are equal.  The $x_i$ are referred to as {\em cube digits}, the $k$ as {\em the cycle digit}.   Edges are of two types:  the first kind, termed a {\em cycle edge}, connects two vertices whose cube digits are equal but whose cycle digit differs by 1 (where we consider the cycle digits of 1 and $n$ as differing by 1, cyclically).  Thus cycle edges connect consecutive vertices in the same cycle.  The second kind of edge, termed a {\em cube edge}, connects two vertices whose cycle digits are equal to a common value $k$, and whose cube digits only differ in the $k$-th digit $x_k$.  Thus cube edges connect adjacent cycles in such a way that $CCC_n$ is 3-regular.

Important from the outset is the observation that the labeling of $CCC_n$ is also not unique.  In particular, we have the following lemma.

\begin{lemma}
\label{lem:label}
A choice of a vertex $u$ in $CCC_n$ as $000\cdots01$, along with a choice of orientation for the cycle containing $u$, determines the labeling of all vertices in $CCC_n$.
\end{lemma}

\begin{proof}
Since $CCC_n$ is a Cayley graph, it is vertex-transitive, so we may label any vertex $u$ as $000\cdots01$.  Once we then choose an orientation for the cycle containing $u$, that specifies the labeling of all other vertices in that cycle as $000\cdots02$ through $000\cdots0n$.  Then by the cube edge relations, this determines the cube digits of all vertices in each neighboring cycle; specifically, the vertices in the cycle connected by a cube edge to $000\cdots0k$ will have a 1 in the $k$-th cube digit and $0$'s in the other cube digits.  As a result, by Lemma \ref{lem:labelcube}, this in fact determines the cube digits for all cycles, hence cube digits for all vertices in $CCC_n$.  Finally, given any vertex in $CCC_n$, it is then connected by a cube edge to another vertex, and they differ in some $k$-th cube digit, thus determining their cycle digit to be $k$; hence all labels for all vertices in $CCC_n$ are then fixed.
\end{proof}

We then have two propositions which follow from Lemma \ref{lem:label}.

\begin{proposition}
\label{prop:aut}
$|Aut(CCC_n)| = n2^{n+1}$.
\end{proposition}

\begin{proof}
The automorphisms of $CCC_n$ are in one-to-one correspondence with distinct vertex labelings of $CCC_n$.  Since there are $n2^n$ vertices in $CCC_n$, there are $n2^n$ choices for the vertex $000\cdots01$, and then there are 2 choices of orientation for the cycle containing that vertex.  Hence by Lemma \ref{lem:label} there are $n2^{n+1}$ different vertex labelings of $CCC_n$ and the proposition follows.
\end{proof}

\begin{proposition}
\label{prop:swap}
Given an edge $uv$ in $CCC_n$, there is an automorphism $\varphi \in Aut(CCC_n)$ such that $\varphi(u)=v$ and $\varphi(v)=u$.
\end{proposition}

\begin{proof}
If $uv$ is a cube edge, then by Lemma \ref{lem:label} there is a labeling of $CCC_n$ (actually two) which labels $u=000\cdots01$ and $v=100\cdots01$; but there is also a labeling of $CCC_n$ which reverses this, labeling $u=100\cdots01$ and $v=000\cdots01$.  Thus there is an automorphism $\varphi$ which matches up vertices with the same digits under these two labelings, so that $\varphi(u)=v$ and $\varphi(v)=u$.  

On the other hand, if $uv$ is a cycle edge, then by Lemma \ref{lem:label} there is a labeling of $CCC_n$ which labels $u=000\cdots01$ and $v=000\cdots0n$; but there is also a labeling of $CCC_n$ which reverses this, labeling $u=000\cdots0n$ and $v=000\cdots01$.  Thus again there is an automorphism $\varphi$ which matches up vertices with the same digits under these two labelings, so that $\varphi(u)=v$ and $\varphi(v)=u$.  
\end{proof}

We now have our first major result.

\begin{proposition}
\label{prop:db}
$CCC_n$ is distance-balanced for all $n$.
\end{proposition}

\begin{proof}
By Proposition 2.4 in \cite{2}, if $G$ is a graph such that for any edge $uv$ in $G$, there is an automorphism $\varphi \in Aut(G)$ that results in $\varphi(u)=v$ and $\varphi(v)=u$, then $G$ is distance-balanced.  By Proposition \ref{prop:swap} above, $CCC_n$ is such a graph, and hence $CCC_n$ is distance-balanced.
\end{proof}

To investigate the nice distance-balancing of $CCC_n$, we need a new method of understanding minimum paths and distance, which we now introduce in the next section.

\section{Interval-with-ends diagrams for paths in $CCC_n$}
\label{sec:one}

A path in $CCC_n$ between two vertices is a sequence of cube and cycle edges.  Whenever an edge is traversed, exactly one digit changes in the corresponding vertex string; if it is a cycle edge, the cycle digit will change by 1, either increasing or decreasing (in the cyclic ordering of $1,\dots,n$); and if it is a cube edge, a cube digit will change from a 1 to 0, or vice versa.  More specifically, if we are at a vertex $x_1x_2 \cdots x_nk$, we only have three directions we can go, namely we can either change the cycle digit $k$ to $k-1$ or $k+1$ (in the cyclic ordering of $1,\dots,n$), and traverse a cycle edge; or we can change the $k$-th cube digit $x_k$ and traverse that cube edge.  With this in mind, our first goal in this section is to establish a new way to represent paths in $CCC_n$.   

By vertex-transitivity, we may focus on paths from $u = 00\cdots01$ to another vertex $x = x_1x_2 \dots x_nk$.  We begin with an informal discussion of a path from $u$ to $x$ before formalizing our observations.  In order to move from $u = 00\cdots01$ to $x = x_1x_2 \cdots x_nk$ via a path $\gamma$, we need to examine which cube digits in $x$ are 1's, for these will need to be changed.  For example, if $x_i=1$, then we know that $\gamma$ will need to traverse cycle edges within some cycle until the cycle digit changes to $i$; then the changing of the $i$-th cube digit from 0 to 1 will result in the traversal of a cube edge, with the result that $x_i =1$.  We then have to do this for all such cube digits in $x$ which are $1$, and then finally move within the last cycle until the cycle digit changes to $k$.

As a result, we can represent paths $\gamma$ in $CCC_n$ symbolically using what we call an {\em interval-with-ends} (IWE) diagram, with four examples shown in Figure \ref{fig:intervalwithends} for IWE diagrams  from $00000001$ to $01011006$ in $CCC_7$.  To construct an IWE diagram from $u =00\cdots01$ to $x = x_1x_2 \cdots x_nk$, we do the following:

\begin{enumerate}
\item Form a loop connecting the numbers $1$ through $n$; these represent the cycle digits for vertices in the path.
\item Circle all of the numbers $i$ where the $i$-th cube digit needs to be changed from $0$ to $1$ in moving from $u$ to $x$.
\item Choose an interval along the loop that includes all the circled digits, has endpoints at circled digits, and has angular support less than $2\pi$; if there are $m$ circled digits, we have $m$ possible intervals.
\item Since we begin at $u$ with cycle digit 1, draw an arrow from $1$ to one of the endpoints of the interval, along the loop, with angular support less than $2\pi$.  We call this arrow the {\em initial end}; there are four possibilities for it, depending on which endpoint of the interval it connects to, and if it proceeds clockwise or counterclockwise.
\item Since we end at $x$ with cycle digit $k$, draw an arrow from the other endpoint of the interval to $k$, along the loop, with angular support less than $2\pi$.  We call this arrow the {\em terminal end}; once the initial end is fixed, there are two possibilities for the terminal end, depending if it proceeds clockwise or counterclockwise.
\end{enumerate}

We note that the interval may be an empty interval in the case where both $u$ and $x$ are in the same cycle in $CCC_n$, in which case we then have a single end connecting 1 to $k$ as our IWE diagram.

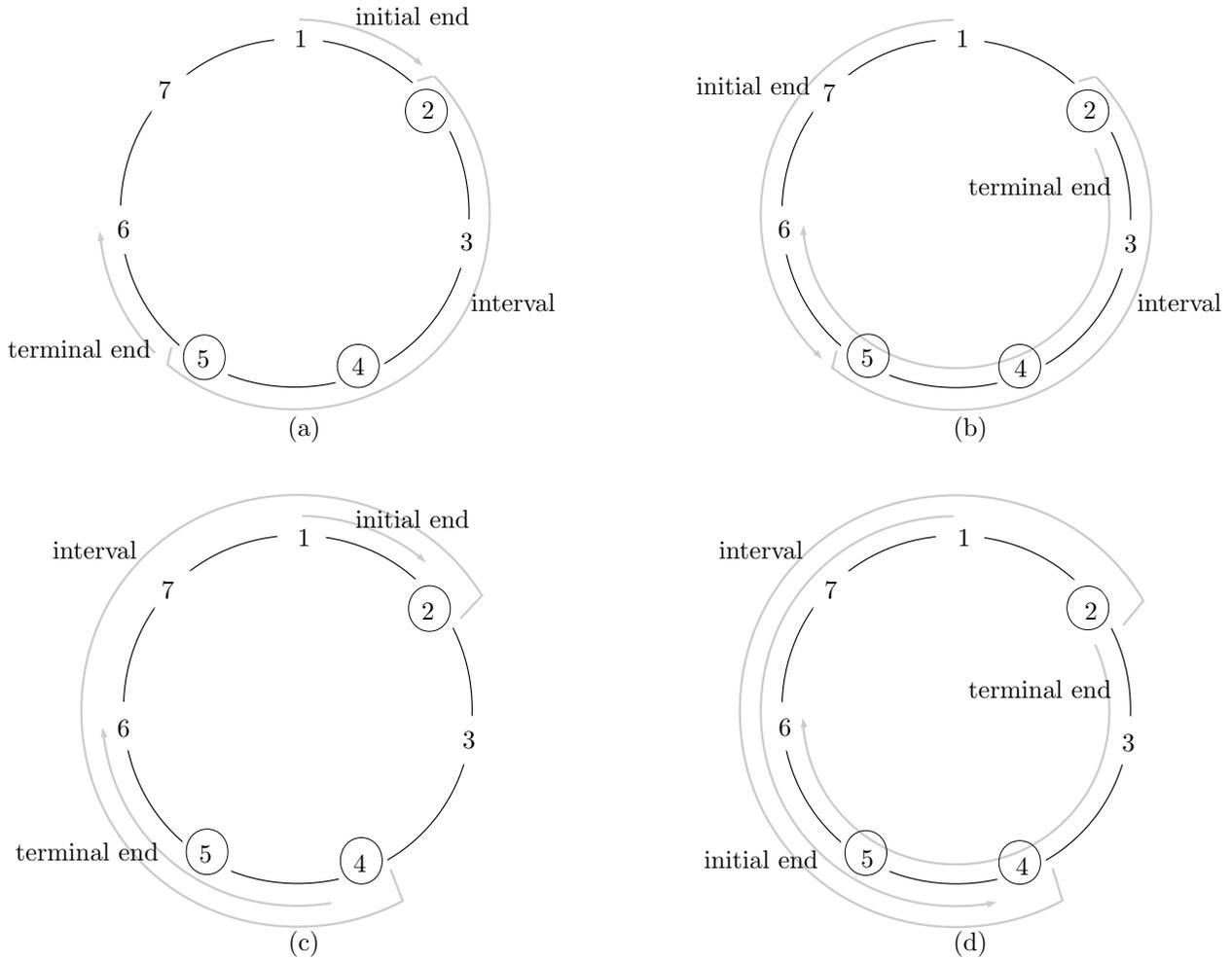
\begin{figure}[h]
\labellist
\small\hair 2pt
\pinlabel $7$ at 124 448
\pinlabel $6$ at 67 269
\pinlabel $5$ at 173 106
\pinlabel $4$ at 372 94
\pinlabel $3$ at 513 254
\pinlabel $2$ at 460 420
\pinlabel $1$ at 300 515
\pinlabel interval at 30 500
\pinlabel {initial end} at 440 540
\pinlabel {terminal end} at 20 110
\pinlabel {(c)} at 300 -10

\pinlabel $7$ at 120 1094
\pinlabel $6$ at 67 914
\pinlabel $5$ at 170 746
\pinlabel $4$ at 370 736
\pinlabel $3$ at 510 897
\pinlabel $2$ at 460 1067
\pinlabel $1$ at 296 1160
\pinlabel interval at 570 820
\pinlabel {initial end} at 440 1190
\pinlabel {terminal end} at 10 760
\pinlabel {(a)} at 300 655

\pinlabel $7$ at 978 1090
\pinlabel $6$ at 920 914
\pinlabel $5$ at 1027 747
\pinlabel $4$ at 1225 734
\pinlabel $3$ at 1367 895
\pinlabel $2$ at 1314 1067
\pinlabel $1$ at 1150 1160
\pinlabel interval at 1430 820
\pinlabel {terminal end} at 1250 970
\pinlabel {initial end} at 880 1100
\pinlabel {(b)} at 1160 655

\pinlabel $7$ at 980 447
\pinlabel $6$ at 921 270
\pinlabel $5$ at 1027 100
\pinlabel $4$ at 1227 90
\pinlabel $3$ at 1364 250
\pinlabel $2$ at 1316 420
\pinlabel $1$ at 1152 515
\pinlabel interval at 890 500
\pinlabel {initial end} at 890 100
\pinlabel {terminal end} at 1250 320
\pinlabel {(d)} at 1160 -10

\endlabellist
\centering
\includegraphics[scale=0.30]{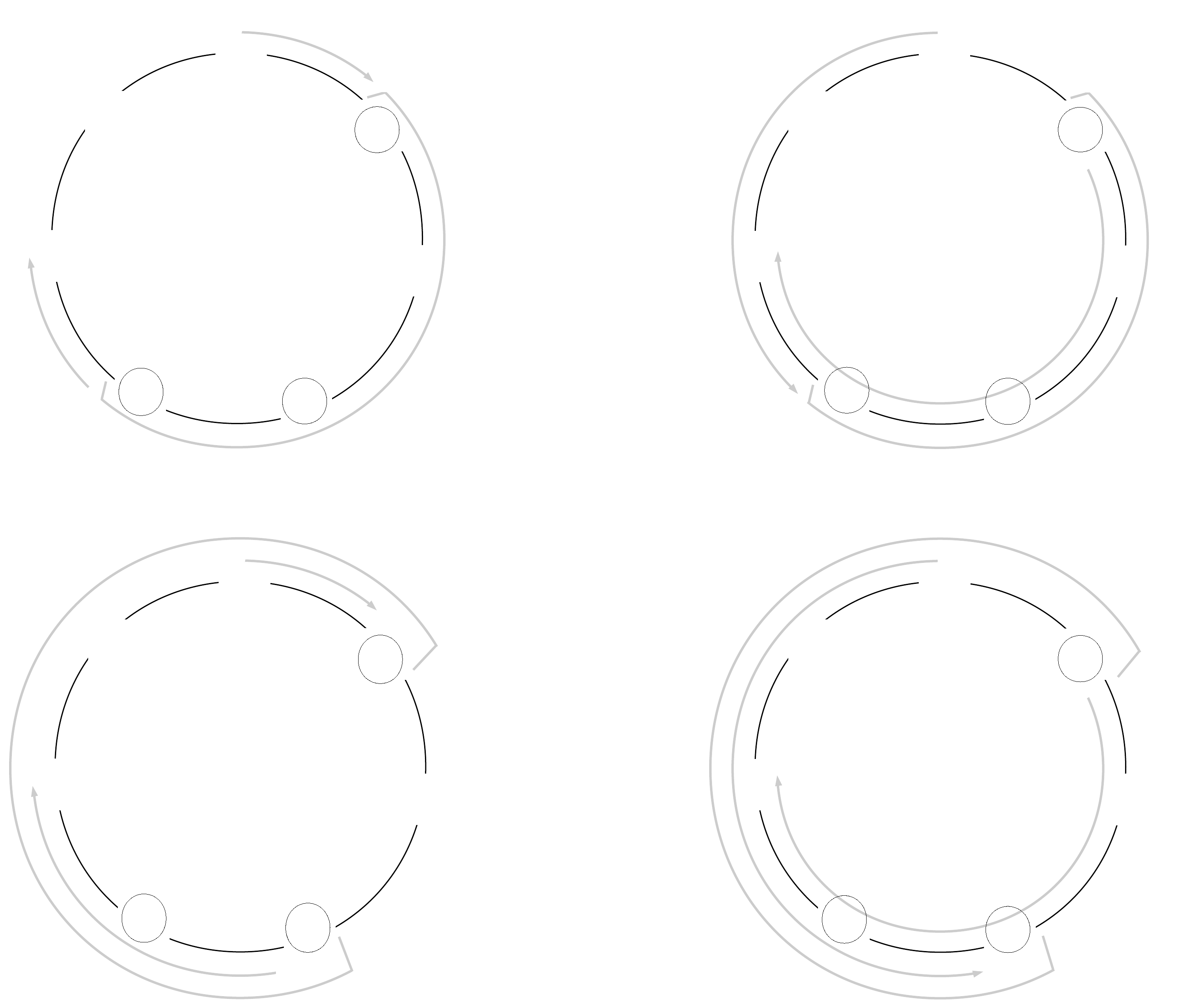}
\caption{Four examples of interval-with-ends (IWE) diagrams from $00000001$ to $01011006$ in $CCC_7$.  Part (a) has length 8, since three circles are included in the interval, and five edges in the loop are included in the interval-with-ends; part (b) has length 13; part (c) has length 11; and part (d) has length 16.}
\label{fig:intervalwithends}
\end{figure}

There are multiple possible IWE diagrams from $u$ to $x$; yet for a given IWE diagram, there is a shortest path $\gamma$ in $CCC_n$ from $u$ to $x$ supported by that IWE diagram, namely the following:  First, we proceed monotonically along the initial end from $1$ to the first endpoint of the interval, and each edge we traverse in the initial end corresponds to a cycle edge in $\gamma$.  Then, at the endpoint of the interval, we have a circled digit, and that circle corresponds to traversing a cube edge in $\gamma$.  We then proceed monotonically along the interval, and every circle we encounter represents a cube edge in $\gamma$, with every edge in the interval a cycle edge in $\gamma$, until we reach the end of the interval with our final circle and cube edge.  Then, we proceed monotonically along the terminal end to $k$, and every edge in the terminal end corresponds to a cycle edge in $\gamma$.  Thus given an IWE diagram $D$, we will denote its length by $|D|$, with

\begin{eqnarray*}
|D| &=& \textrm{(\# of edges in initial end)} + \textrm{(\# of edges in terminal end)} \\ &\ &+ \textrm{\ (\# of edges in interval)} + \textrm{(\# of circled digits)}
\end{eqnarray*}

As indicated in Figure \ref{fig:intervalwithends}, the length of $\gamma$ for the IWE diagram in part (a) is 8, for part (b) is 13, for part (c) is 11, and for part (d) is 16.  Note that if we do not move monotonically along the ends of the interval, but double back, or change cube digits other than the circled ones, this can give us other paths in $CCC_n$.   But these will necessarily be longer than $|D|$.  Furthermore, any path in $CCC_n$ between $u =00\cdots01$ and $x = x_1x_2 \cdots x_nk$ can be mapped out in this cycle digit loop, where we record the edges traversed in the cycle digit loop, and circle the cube digits changed; although this mapping may not be an IWE itself, nevertheless an IWE diagram $D$ will be a subset of it, and hence the length of that path will be at least $|D|$.  Therefore, if $u =  00\cdots01$ and $x = x_1x_2 \cdots x_nk$, we have that in $CCC_n$, $$d(u,x) = \min\{|D| : D \textrm{\ is an IWE diagram from\ } u \textrm{\ to\ } x\}.$$

\noindent Thus $d(u,x) = |D|$ for at least one IWE diagram $D$, which gives a shortest path $\gamma$ in $CCC_n$ from  $u$ to $x$.

\section{Nice distance-balancing in $CCC_n$}
\label{sec:two}

In order to determine under what conditions $CCC_n$ may be nicely distance-balanced, we first consider cube edges.  We recall that in $CCC_n$ there are a total of $n2^n$ vertices.

\begin{proposition}
\label{prop:cubedge}
If $uv$ is a cube edge in $CCC_n$, then $|W_{u,v}| = |W_{v,u}| = n2^{n-1}$.
\end{proposition}

\begin{proof}
We may assume that $u = 00 \cdots 01$ and $v = 10 \cdots 01$.  We claim that any vertex with a 0 as its first cube digit is in $W_{u,v}$, and any vertex with a 1 as its first cube digit is in $W_{v,u}$.  To see that $x=1x_2 \cdots x_nk \in W_{v,u}$, consider an IWE diagram $D$ from $u$ to $x$ such that $|D| = d(u,x)$.  Then observe that since $u$ and $v$ differ only in the first cube digit, an IWE diagram $D'$ from $v$ to $x$ can be obtained by using $D$, but deleting the circle around 1.  Thus $d(x,v) \leq d(u,x)-1 < d(u,x)$, so that $x \in W_{v,u}$.  An entirely similar argument shows that if $x=0x_2 \cdots x_nk$, then $x \in W_{u,v}$, and this proves the proposition.
\end{proof}

\begin{remark}
\label{rem}
Since $|W_{u,v}| = |W_{v,u}| = n2^{n-1}$ we also have that $W^v_u = \emptyset$.
\end{remark}

We now turn to cycle edges for the following proposition.

\begin{proposition}
\label{prop:cycleedge}
If $uv$ is a cycle edge in $CCC_n$, then $W^v_u = \emptyset$ if and only if $n$ is even.
\end{proposition}

\begin{proof}
We may assume throughout the proof that $u = 00\cdots01$ and $v=00\cdots0n$. 

When $n$ is odd, consider the vertex $x=00\cdots0\left(\frac{n+1}{2}\right)$.  It is in the same cycle as both $u$ and $v$, and so the IWE diagram $D_u$ that realizes $d(x,u)$ is a single end that connects $1$ to $\left(\frac{n+1}{2}\right)$, with length $\frac{n-1}{2}$.  Similarly, the IWE diagram $D_v$ that realizes $d(x,v)$ is a single end that connects $n$ to $\left(\frac{n+1}{2}\right)$, with length $\frac{n-1}{2}$.  Thus $d(x,u)=d(x,v)$ in $CCC_n$, with $x \in W_u^v$.  We therefore conclude that $W^v_u \neq \emptyset$ when $n$ is odd.

In general, we now consider an arbitrary $x \in W^v_u$ in $CCC_n$, and we will show that $n$ must be odd; this will conclude the proof.  To set notation, we assume that $x = x_1x_2\cdots x_nk$.  We first consider the case where $x_i=0$ for all $i \in \{1,\dots,n\}$, so that $x$ lies in the same cycle as both $u$ and $v$.  Then the IWE diagrams realizing $d(x,u)$ and $d(x,v)$ are just single ends connecting $u$ to $x$ and $v$ to $x$, respectively.  Since $x \in W^v_u$ with $d(x,u)=d(x,v)$, we must have $d(x,u) = k-1 = d(x,v) = n-k$, so that $n+1 = 2k$; thus $n$ must be odd.

We now consider the case where $x$ does not lie in the same cycle as $u$ and $v$, and we observe that the cube digits that need to be changed to 1 in moving from $u$ to $x$ are the same cube digits that need to be changed in moving from $v$ to $x$.  We fix an IWE diagram $D$ realizing $d(x,u)$, and observe that its interval $I$ will have endpoints at numbers $i$ and $j$ which represent cube digits that need to be changed.  An example is shown in Figure \ref{fig:prealgebraic}, where the specific cyclic ordering of $i,j,k$ will not enter into the argument.  Likewise we only assume that of $1,i,j,k$ and $n$, only $1$ and $n$ are necessarily distinct.

\begin{figure}[h]
\labellist
\small\hair 2pt

\pinlabel $i$ at 468 400
\pinlabel $n$ at 200 500
\pinlabel $1$ at 350 500
\pinlabel $I$ at 20 300

\pinlabel $k$ at 150 120
\pinlabel $j$ at 370 90

\endlabellist
\centering
\includegraphics[scale=0.3]{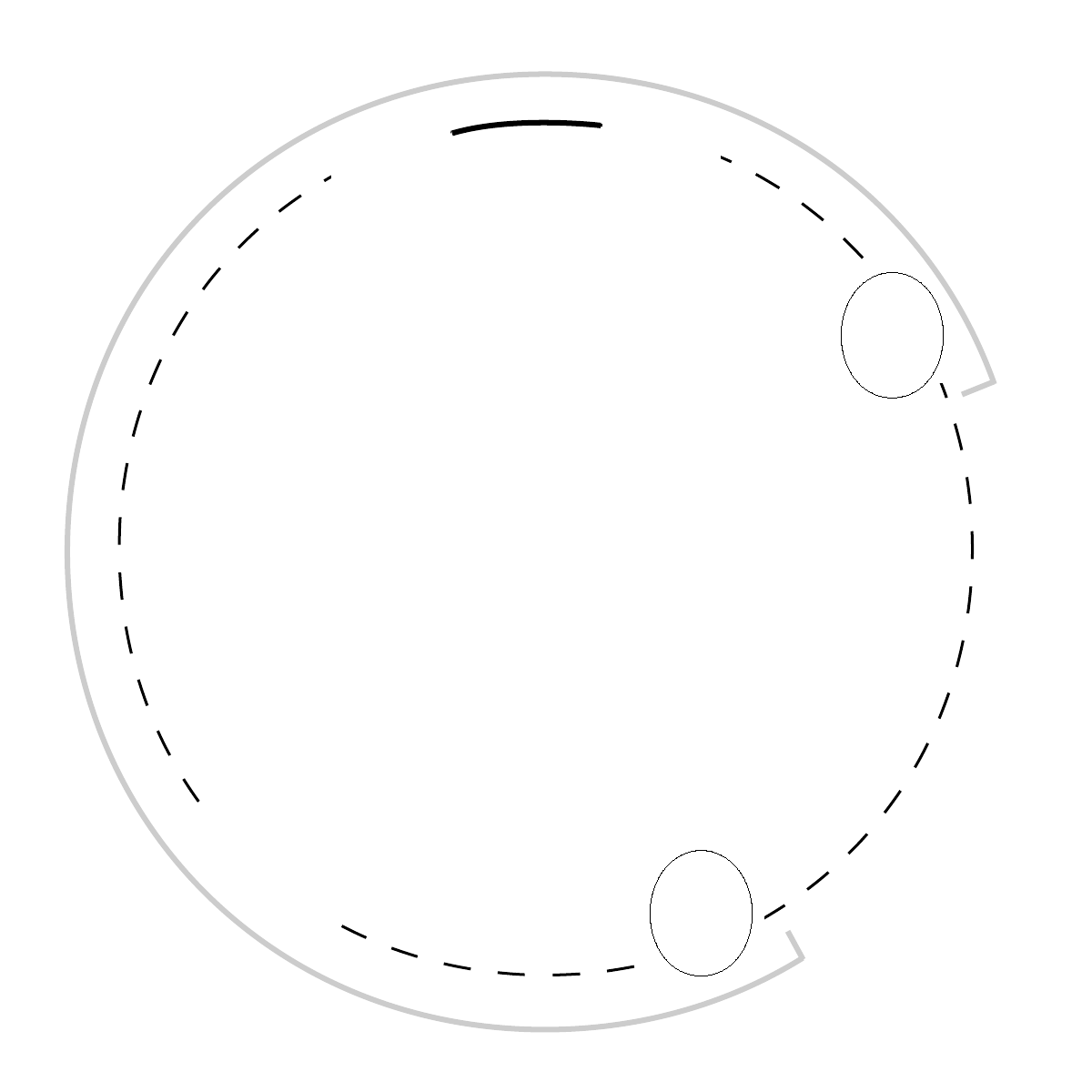}
\caption{An interval $I$ for a possible IWE diagram realizing $d(x,u)$.}
\label{fig:prealgebraic}
\end{figure}
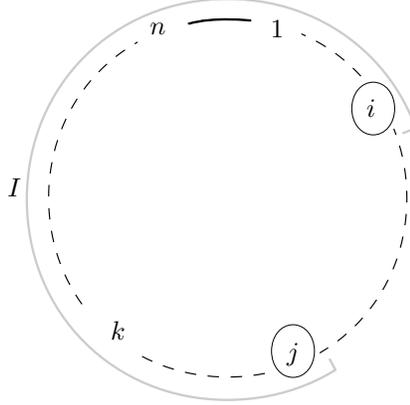

Our first observation is that no end emanating from $1$ to either $i$ or $j$ can proceed counterclockwise, for if it did, this IWE diagram could work for $v$, but with one less edge, hence resulting in $d(x,v) < d(x,u)$; but this is not true, since we are assuming $x \in W^v_u$.  So any end emanating from $1$ for an IWE diagram realizing $d(x,u)$ must proceed clockwise.  Likewise, any end emanating from $n$ for an IWE diagram realizing $d(x,v)$ must proceed counterclockwise.  

We now consider the case where the IWE diagram $D_u$ realizing $d(x,u)$ has the same interval as the IWE diagram $D_v$ realizing $d(x,v)$.  In this case, denote the number of edges in the initial and terminal ends for $D_u$, with an endpoint at $1$ and $k$ respectively, as $\alpha_u^1$ and $\alpha_u^k$.  Likewise denote the number of edges in the initial and terminal ends for $D_v$, with an endpoint at $n$ and $k$ respectively, as $\beta_v^n$ and $\beta_v^k$.  Since $D_u$ and $D_v$ share the same interval with the same circled cube digits, and we are assuming $d(x,u)=d(x,v)$, we must have 

\begin{equation}
\label{eqn1}
\alpha_u^1+\alpha_u^k = \beta_v^n + \beta_v^k.
\end{equation}

Observe that each of these quantities has limited possibilities.  Specifically, from the reasoning in the preceding paragraph, $\alpha_u^1$ must be either $i-1$ or $j-1$, and likewise $\beta_v^n$ must be $n-i$ or $n-j$.  The possibilities for $\alpha_u^k$ and $\beta_v^k$ are more numerous, and so we take a moment to focus on $\alpha_u^k$ and observe that the terminal end with endpoint at $k$ has four options:  

\renewcommand\labelitemi{\tiny$\bullet$}

\begin{itemize}
\item if it joins $k$ to $i$ without passing through $n$, then $\alpha_u^k = |k-i|$ and $\alpha_u^1 = j-1$;
\item  if it joins $k$ to $i$ and passes through $n$, then $\alpha_u^k = n-k+i$ and $\alpha_u^1 = j-1$;
\item  if it joins $k$ to $j$ without passing through $n$, then $\alpha_u^k = |k-j|$ and $\alpha_u^1 = i-1$;
\item  if it joins $k$ to $j$ and passes through $n$, then $\alpha_u^k = n-k+j$ and $\alpha_u^1 = i-1$.
\end{itemize}

\noindent Likewise considering $\beta_v^k$, the options for its terminal end with endpoint at $k$ will be:

\begin{itemize}
\item if it joins $k$ to $i$ without passing through $n$, then $\beta_u^k = |k-i|$ and $\beta_u^n = n-j$;
\item  if it joins $k$ to $i$ and passes through $n$, then $\beta_u^k = n-k+i$ and $\beta_u^n = n-j$;
\item  if it joins $k$ to $j$ without passing through $n$, then $\beta_u^k = |k-j|$ and $\beta_u^n = n-i$;
\item  if it joins $k$ to $j$ and passes through $n$, then $\beta_u^k = n-k+j$ and $\beta_u^n = n-i$.
\end{itemize}

The key observation is that regardless of which possibility occurs, the left hand side of Equation \ref{eqn1} will contain a $-1$ term, along with a $\pm i$, $\pm j$, and $\pm k$ term, and possibly one $n$ term; likewise the right hand side of Equation \ref{eqn1} will contain $n$, along with a $\pm i$, $\pm j$, and $\pm k$ term, and possibly one more $n$ term.  The result is that when we isolate the $n$'s and 1 on the right hand side, we obtain an even left hand side with two terms each of $i, j$ and $k$.  Therefore, in order for $n$ to have an integer solution, $n$ must be odd.

We conclude by considering the case where the IWE diagram $D_u$ realizing $d(x,u)$ does not use the same interval as the IWE diagram $D_v$ realizing $d(x,v)$.  This case does not apply when there is just one circled cube digit that needs to be changed, and so from here on we may assume that $i\neq j$.  For this case denote the endpoints of the interval for $D_v$ as $i'$ and $j'$, and the interval itself as $I'$, where we likewise necessarily have $i' \neq j'$.  An example is shown in Figure \ref{fig:overlapping}, where the two intervals are colored in light gray and dark gray, respectively.

\begin{figure}[h]
\labellist
\small\hair 2pt

\pinlabel $i'$ at 100 390

\pinlabel $i$ at 470 400
\pinlabel $n$ at 200 500
\pinlabel $1$ at 350 500
\pinlabel $k$ at 490 205

\pinlabel $j'$ at 145 120
\pinlabel $j$ at 370 90

\pinlabel $I$ at 20 250
\pinlabel $I'$ at 570 250

\endlabellist
\centering
\includegraphics[scale=0.3]{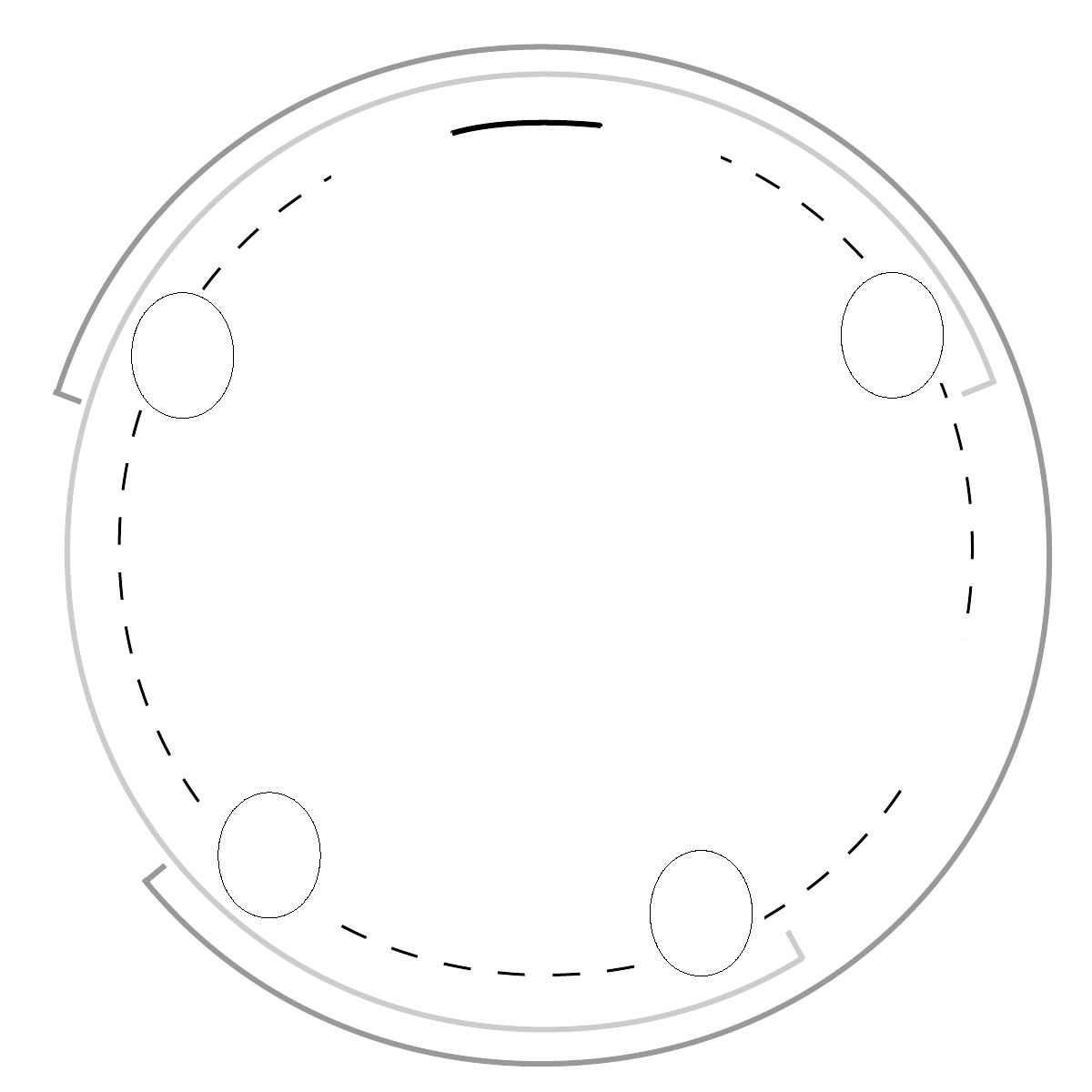}
\caption{Intervals $I$ and $I'$ for possible IWE diagrams realizing $d(x,u)$ and $d(x,v)$, respectively.}
\label{fig:overlapping}
\end{figure}
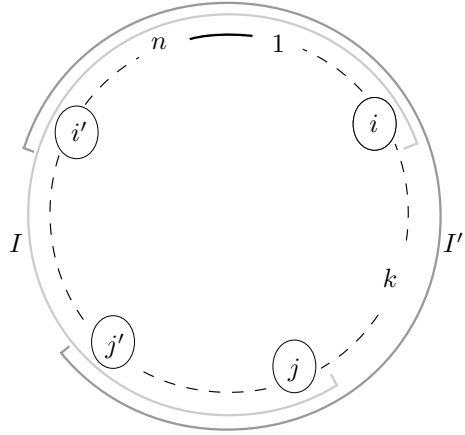

We now observe that the total number of edges in $I$ will be $n - |i-j|$, and the total number of edges in $I'$ will be $n - |i'-j'|$.  As above, denote the number of edges in the initial and terminal ends for $D_u$ as $\alpha_u^1$ and $\alpha_u^k$, respectively.  Likewise denote the number of edges in the initial and terminal ends for $D_v$ as $\beta_v^n$ and $\beta_v^k$, respectively.  Since $D_u$ and $D_v$ share the same circled cube digits, and we are assuming $d(x,u)=d(x,v)$, we must have

$$n - |i-j|+ \alpha_u^1+\alpha_u^k =n - |i'-j'|+ \beta_v^n + \beta_v^k$$

\noindent or 

\begin{equation}
\label{eqn2}
-|i-j|+ \alpha_u^1+\alpha_u^k =-|i'-j'|+ \beta_v^n + \beta_v^k.
\end{equation}

As above, the $\alpha_u^1+\alpha_u^k$ terms will contain a $-1$ term, along with a $\pm i$, $\pm j$, and $\pm k$ term, and possibly one $n$ term; similarly, the $\beta_v^n + \beta_v^k$ terms will contain $n$, along with a $\pm i'$, $\pm j'$, and $\pm k$ term, and possibly one more $n$ term.  As a result, now the left hand side of Equation \ref{eqn2} has two terms each of $i$ and $j$, and one term of $k$, and the right hand side has two terms each of $i'$ and $j'$, and one term of $k$.  Moving all these terms to the left hand side, and all $n$'s and the 1 to the right hand side, we see the left hand side is even with two terms each of $i,i',j,j'$ and $k$.  Thus, again in order for $n$ to have an integer solution, $n$ must be odd.
\end{proof}

Our main theorem then follows.

\begin{theorem}
\label{thm:ndb}
$CCC_n$ is nicely distance-balanced if and only if $n$ is even.
\end{theorem}

\begin{proof}
We know by Proposition \ref{prop:cubedge} that if $uv$ is a cube edge, $|W_{u,v}| = |W_{v,u}| = n2^{n-1}$, and by Remark \ref{rem} we know $W_u^v = \emptyset$.  But if $uv$ is a cycle edge and $n$ is odd, by Proposition \ref{prop:cycleedge} we know that $W^v_u \neq \emptyset$.  Applying now Proposition \ref{prop:db}, we must have $|W_{u,v}| = |W_{v,u}| < n2^{n-1}$.  Hence for $n$ odd, $CCC_n$ is not nicely distance-balanced.

On the other hand, if $n$ is even and $uv$ is a cycle edge, then by Proposition \ref{prop:cycleedge} we know that $W^v_u = \emptyset$, so by Proposition \ref{prop:db} we must have $|W_{u,v}| = |W_{v,u}| = n2^{n-1}$.  Hence for $n$ even, $CCC_n$ is nicely distance-balanced.
\end{proof}

\begin{remark}
As distance-regular graphs are nicely distance-balanced, Theorem \ref{thm:ndb} proves that the graphs $CCC_n$, where $n$ is odd, cannot be distance-regular (see \cite{6} for a definition and background on distance-regular graphs).
\end{remark}

\end{document}